\documentclass[11pt]{article}
\usepackage[letterpaper,margin=1in]{geometry}
\usepackage{amsmath,amssymb,amsthm,mathtools}
\usepackage{mathrsfs}
\usepackage[dvipsnames]{xcolor}
\usepackage{graphicx}
\usepackage{microtype}
\usepackage{tikz}
\usetikzlibrary{arrows.meta,positioning,calc,fit,decorations.pathreplacing}
\usepackage{enumitem}
\usepackage{array}
\setlist{nosep,leftmargin=*}
\usepackage{float}
\usepackage{hyperref}
\hypersetup{
  colorlinks=true,
  linkcolor=blue!55!black,
  citecolor=blue!55!black,
  urlcolor=blue!55!black
}

\newtheorem{theorem}{Theorem}[section]
\newtheorem{lemma}[theorem]{Lemma}
\newtheorem{proposition}[theorem]{Proposition}
\newtheorem{corollary}[theorem]{Corollary}

\newtheorem{problem}[theorem]{Problem}

\theoremstyle{definition}
\newtheorem{definition}[theorem]{Definition}
\newtheorem{remark}[theorem]{Remark}
\newtheorem{example}[theorem]{Example}

\usepackage[nameinlink,capitalize]{cleveref}

\crefname{theorem}{Theorem}{Theorems}
\crefname{lemma}{Lemma}{Lemmas}
\crefname{proposition}{Proposition}{Propositions}
\crefname{corollary}{Corollary}{Corollaries}
\crefname{definition}{Definition}{Definitions}
\crefname{remark}{Remark}{Remarks}
\crefname{example}{Example}{Examples}
\crefname{problem}{Problem}{Problems}
\Crefname{theorem}{Theorem}{Theorems}
\Crefname{lemma}{Lemma}{Lemmas}
\Crefname{proposition}{Proposition}{Propositions}
\Crefname{corollary}{Corollary}{Corollaries}
\Crefname{definition}{Definition}{Definitions}
\Crefname{remark}{Remark}{Remarks}
\Crefname{example}{Example}{Examples}
\Crefname{problem}{Problem}{Problems}

\newcommand{\dist}{\operatorname{dist}}

\begin{document}

\title{Exact Leaf Powers on Cycles, Ladders, Crowns, and Multipartite Block Graphs}
\author{Peng Li\thanks{School of Mathematics and Statistics, Hanjiang Normal University, Shiyan 442000, China. Email: \texttt{lipeng1@hjnu.edu.cn}.}
\and Yangjing Long\thanks{School of Mathematics and Statistics, Central China Normal University, 152 Luoyu Road, Wuhan 430079, China. Email: \texttt{yangjing@mail.ccnu.edu.cn}. Corresponding author.}}
\date{}
\maketitle

\begin{abstract}
Exact $k$-leaf powers are graphs whose edges are exactly the pairs of leaves at distance $k$ in a tree.  We study exact leaf powers on several natural structural test classes which probe how induced squares, dense bipartite graphs, and multipartite blocks can be represented at fixed tree distances.  The common tool is the support-tree reduction: after deleting the graph leaves, exact $k$-leaf powers are precisely support-tree distance-$(k-2)$ graphs.

Our most detailed root-classification theorem concerns chordless cycles: we give a complete terminal block-language description of all exact $5$-leaf roots of $C_l$, $l\ge8$; the language has a normal terminal-$23$ part and a complementary terminal-$123$ part.  The other exact-distance-five results place this benchmark in a broader structural landscape.  We prove the sharp ladder dichotomy
\[
 L_t\text{ is an exact }5\text{-leaf power}\quad\Longleftrightarrow\quad t\le2,
\]
showing that sparse chains of induced squares are severely limited.  In contrast, dense bipartite square structures are often representable: among graphs whose blocks are complete multipartite, the exact $5$-leaf powers are exactly the bipartite members, and every bipartite co-cluster graph is an exact $k$-leaf power for every $k\ge5$.  Thus all crown graphs $K_{n,n}-M$ and the cube $Q_3$ occur at exact distance five.

We also obtain larger-distance consequences.  Complete multipartite graphs are exact $k$-leaf powers exactly when $k$ is even or the graph is bipartite.  Block-complete multipartite graphs are classified for all odd $k\ge5$ and represented for all even $k\ge8$; at the exceptional boundary $k=6$ we prove a sufficient orientation theorem and a complete theorem for multipartite block fans.  Stable blow-ups preserve exact $k$-leaf representability for every $k\ge3$, giving modular versions of the subclass classifications.
\end{abstract}

\noindent\textbf{Keywords:} exact leaf power; chordless cycle; ladder graph; crown graph; complete multipartite graph; support tree.

\noindent\textbf{2020 Mathematics Subject Classification:} 05C05; 05C38; 05C75; 05C85.

\bigskip

\section{Introduction}

Leaf powers encode graph adjacencies by distances between leaves of a tree.  In an ordinary $k$-leaf power, two leaves are adjacent when their distance is at most $k$; in an exact $k$-leaf power, adjacency is prescribed by the single value $k$.  Ordinary leaf powers were introduced by Nishimura, Ragde and Thilikos \cite{NRT2002}.  The cases $k=3$ and $k=4$ have structural and linear-time recognition theorems \cite{BLS2006,BLS2008}, and Lafond proved polynomial-time recognition for ordinary $k$-leaf powers for every fixed $k$ \cite{Lafond2022,Lafond2023}.  Recent work has also considered restricted root shapes \cite{BergougnouxHogemoTelleVatshelle2022}, optimal leaf roots \cite{LeRosenke2024}, leaf-rank lower bounds \cite{Hogemo2024}, forbidden-induced-subgraph limitations for ordinary $k$-leaf powers with $k\ge5$ \cite{DupreLafondNdiayeVetta2025}, and the complexity of unrestricted leaf-power and pairwise-compatibility-graph recognition \cite{DupreLafondNdiaye2026}.  Pairwise compatibility graphs form a related phylogenetic framework; see \cite{CalamoneriSinaimeri2016,CalamoneriMontiSinaimeri2025}.  A survey treatment of ordinary leaf powers appears in \cite{RosenkeLeBrandstaedt2021}.

Exact leaf powers were introduced by Brandstaedt, Le and Rautenbach \cite{BLR2010}.  Exactness changes the nature of the representation: one asks for a sphere in the leaf metric rather than a threshold ball.  Thus odd exact distances force bipartiteness at the support level, whereas even exact distances can realize multipartite behaviour.  The first exact cases, $k=3$ and $k=4$, are well understood \cite{BLR2010}; exact distance five is the first case in which cycles, induced squares, complete bipartite blocks, and nonunique roots interact in a nontrivial way.

This paper develops explicit structure theorems for exact leaf powers on representative graph families, with exact distance five as the main case.  The most detailed root-classification benchmark is the cycle theorem: for every chordless cycle $C_l$, $l\ge8$, all exact $5$-leaf roots are described by an explicit terminal block language.  The other results should be read as a structural landscape around this benchmark, showing which square patterns and multipartite blocks are representable and where the first larger-distance boundary occurs.  The common tool is a support-tree model.  If $T$ is an exact $k$-leaf root and the graph leaves are deleted, then each graph vertex $v$ has a support vertex $s_v$ in a tree $R$, and
\[
 uv\in E(G)\quad\Longleftrightarrow\quad \dist_R(s_u,s_v)=k-2.
\]
For $k=5$ this is a support-distance-three representation.  The results below show how the same simple metric mechanism leads to both obstructions and constructions.

The graph classes considered here are not isolated examples.  They form a small structural test set for exact-distance tree representations.  The chordless-cycle theorem is the most detailed root-level benchmark in the paper: it treats the basic nonchordal case and gives a finite terminal block language for all exact $5$-roots.  The other subclass theorems explain where this benchmark sits in the larger landscape.  Ladders are sparse chains of induced squares; they are rigid enough that only the first two ladders are representable.  Complete bipartite blocks and bipartite co-cluster graphs are dense bipartite families; despite containing many induced squares, they have compact support-tree gadgets.  Finally, complete multipartite block graphs reveal how the parity of the exact distance changes the picture, with an exceptional boundary at $k=6$.

\begin{center}
\small
\renewcommand{\arraystretch}{1.08}
\begin{tabular}{@{}>{\raggedright\arraybackslash}p{.30\textwidth}>{\raggedright\arraybackslash}p{.43\textwidth}>{\raggedright\arraybackslash}p{.20\textwidth}@{}}
\hline
Class & Structural role & Result type \\
\hline
Chordless cycles & basic nonchordal test objects & root language \\
Ladders $L_t$ & minimal chains of overlapping squares & threshold $t\le2$ \\
Complete-bipartite block graphs & dense square blocks glued at cutvertices & positive class \\
Co-cluster graphs; crowns & dense bipartite graphs with structured nonedges & positive class \\
Complete multipartite graphs & parity test for larger exact distances & all-$k$ dichotomy \\
Exact-six fans & first support-distance-four collision boundary & fan theorem \\
\hline

\end{tabular}
\end{center}

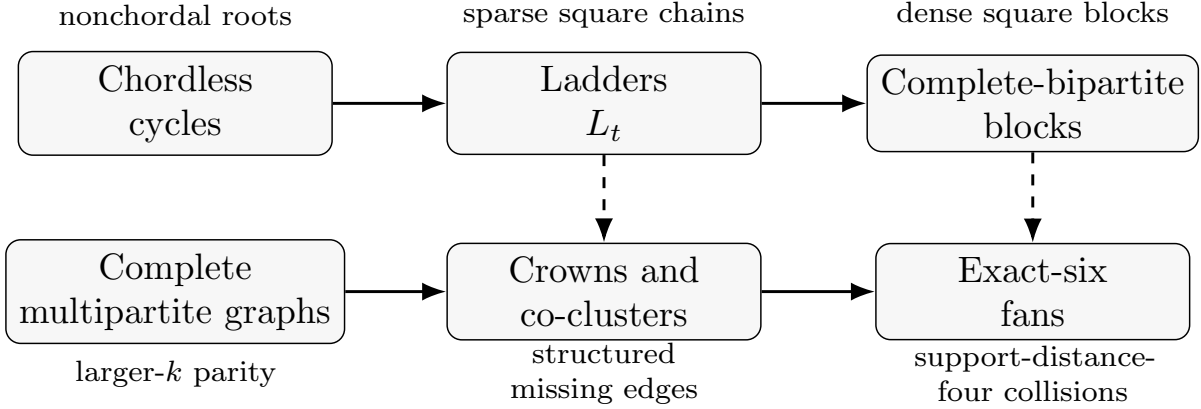
\begin{figure}[H]
\centering
\resizebox{.96\textwidth}{!}{%
\begin{tikzpicture}[>=Latex,
  every node/.style={font=\small},
  box/.style={draw,rounded corners,align=center,minimum width=3.0cm,minimum height=.85cm,fill=gray!6},
  role/.style={align=center,font=\scriptsize,text width=3.0cm}
]
\node[box] (cycles) at (0,1.8) {Chordless\\cycles};
\node[box] (ladders) at (4.1,1.8) {Ladders\\$L_t$};
\node[box] (blocks) at (8.2,1.8) {Complete-bipartite\\blocks};
\node[box] (multi) at (0,0) {Complete\\multipartite graphs};
\node[box] (crowns) at (4.1,0) {Crowns and\\co-clusters};
\node[box] (six) at (8.2,0) {Exact-six\\fans};
\node[role] at (0,2.65) {nonchordal roots};
\node[role] at (4.1,2.65) {sparse square chains};
\node[role] at (8.2,2.65) {dense square blocks};
\node[role] at (0,-.78) {larger-$k$ parity};
\node[role] at (4.1,-.78) {structured missing edges};
\node[role] at (8.2,-.78) {support-distance-four collisions};
\draw[->,thick] (cycles) -- (ladders);
\draw[->,thick] (ladders) -- (blocks);
\draw[->,thick] (multi) -- (crowns);
\draw[->,thick] (crowns) -- (six);
\draw[->,thick,dashed] (ladders) -- (crowns);
\draw[->,thick,dashed] (blocks) -- (six);
\end{tikzpicture}%
}
\caption{The structural test classes used in the paper.  Cycles and ladders probe nonchordal and sparse-square behaviour; complete-bipartite blocks and co-clusters give dense positive families; complete multipartite graphs and exact-six fans expose parity and boundary phenomena for larger exact distances.}
\label{fig:structural-map}
\end{figure}

Our main results are as follows.  The first item is the most detailed root-classification benchmark; the subsequent theorems show how the same exact-distance mechanisms separate sparse square-chain obstructions, dense bipartite positive families, larger-distance parity, and the even boundary at $k=6$.
\begin{enumerate}[label=\textup{(\roman*)},leftmargin=2.2em]
\item \textbf{Chordless cycles.}  This is the root-level classification benchmark of the paper.  We give a complete terminal block language for all exact $5$-leaf roots of $C_l$, $l\ge8$; see \Cref{thm:cycle-block-language}.  The normal terminal-$23$ language corresponds to the standard finite-block spine pattern, while the terminal-$123$ language accounts for the complementary endpoint configurations.

\item \textbf{Ladders and square chains.}  If $L_t$ is the $t$-square ladder, then
\[
 L_t \text{ is an exact }5\text{-leaf power}\quad\Longleftrightarrow\quad t\le2.
\]
The proof is a direct support-tree metric argument.  It also gives the induced square-chain obstruction: no exact $5$-leaf power contains an induced square-chain of length three.

\item \textbf{Multipartite block graphs.}  Within the class of connected graphs whose blocks are complete multipartite, exact distance five singles out the bipartite members:
\[
 G\text{ is an exact }5\text{-leaf power}\quad\Longleftrightarrow\quad G\text{ is bipartite}.
\]
Equivalently, every block is complete bipartite.  The constructive direction is given by one support edge per complete bipartite block, glued along cutvertices.

\item \textbf{Crowns and co-cluster graphs.}  Every bipartite co-cluster graph is an exact $k$-leaf power for every $k\ge5$.  This includes all crown graphs $K_{n,n}-M$ and shows, in particular, that $Q_3$ is an exact $5$-leaf power.

\item \textbf{Larger exact distances.}  Complete multipartite graphs have the all-$k$ classification
\[
 K_{n_1,\ldots,n_s}\text{ is exact }k
 \quad\Longleftrightarrow\quad
 k\text{ is even or }s\le2.
\]
For block-complete multipartite graphs, all odd distances $k\ge5$ behave like distance five, all even distances $k\ge8$ admit the whole class, and $k=6$ has a genuine short-shadow boundary.  We prove a sufficient exact-six orientation theorem and a complete exact-six classification for multipartite block fans.

\item \textbf{Stable blow-ups.}  For every $k\ge3$, stable false-twin blow-ups preserve exact $k$-leaf representability.  Hence the structural classifications above have immediate modular versions.
\end{enumerate}

These classifications are sharp in different directions.  Although every exact odd-leaf power is bipartite, bipartiteness alone is far from sufficient: the three-square ladder is already excluded.  On the other hand, the obstruction is not the mere presence of many induced squares.  Arbitrarily large complete bipartite blocks are representable, and so are crowns and more general bipartite co-cluster graphs.  Thus the sparse ladder obstruction and the dense positive families identify two genuinely different square mechanisms.  The larger-distance results add a parity layer, while the exact-six fan theorem explains why support-distance four is a real boundary: two short same-part shadows can collide across a cutvertex, a phenomenon that disappears in the even constructions for $k\ge8$.

In this sense the cycle theorem supplies the paper's most detailed root-level benchmark, while the remaining theorems turn the same support-distance viewpoint into a structural map of exact leaf powers on several minimal but representative graph families: sparse square chains, dense multipartite blocks, dense bipartite graphs with structured missing edges, and the first even-distance fan boundary.

The paper is organized as follows.  \Cref{sec:support} gives the support-tree reduction and basic closure facts.  \Cref{sec:ladders} proves the ladder dichotomy.  \Cref{sec:block-complete-bipartite,sec:cocluster} treat complete bipartite blocks and bipartite co-cluster graphs.  \Cref{sec:larger-k-blocks} contains the larger-distance multipartite results and the exact-six boundary.  \Cref{sec:blowups-boundaries} records stable blow-ups and odd-distance boundary subclasses.  \Cref{sec:cycle-correction} gives the terminal block-language theorem for chordless cycles.

\section{Support trees and reductions}\label{sec:support}

We work with finite simple undirected graphs.  A graph $G$ is an \emph{exact $k$-leaf power} if there is a tree $T$ whose leaf set is $V(G)$ and such that
\[
 uv\in E(G) \quad\Longleftrightarrow\quad \dist_T(u,v)=k.
\]
Such a tree $T$ is an exact $k$-leaf root.

\begin{definition}[support-distance representations]\label{def:support-distance}
Let $d\ge0$.  A \emph{support-tree distance-$d$ representation} of a graph $H$ is a tree $R$ together with a map
\[
 s:V(H)\to V(R)
\]
such that, for all distinct $u,v\in V(H)$,
\[
 uv\in E(H) \quad\Longleftrightarrow\quad \dist_R(s(u),s(v))=d.
\]
The case $d=3$ is called a support-tree distance-three representation.  If $H$ is false-twin-free, we often identify a vertex $u$ with its support $s(u)$.
\end{definition}

\begin{theorem}[support-distance model]\label{thm:support-model}
Let $k\ge2$.  A graph $G$ is an exact $k$-leaf power if and only if it has a support-tree distance-$(k-2)$ representation.
\end{theorem}

\begin{proof}
Let $T$ be an exact $k$-leaf root of $G$.  Delete all graph leaves from $T$, and let $R$ be the resulting support tree.  If $u\in V(G)$, let $s(u)$ be the unique neighbour of the leaf $u$ in $T$.  Then, for distinct $u,v\in V(G)$,
\[
 \dist_T(u,v)=\dist_R(s(u),s(v))+2.
\]
Thus $uv\in E(G)$ if and only if $\dist_R(s(u),s(v))=k-2$.

Conversely, suppose that $R$ and $s$ form a support-tree distance-$(k-2)$ representation of $G$.  Attach a new leaf $\ell_u$ to $s(u)$ for every $u\in V(G)$, and identify $\ell_u$ with $u$ in the leaf-labelled tree.  Then
\[
 \dist_T(\ell_u,\ell_v)=\dist_R(s(u),s(v))+2,
\]
so the exact distance-$k$ leaf graph of $T$ is precisely $G$.
\end{proof}

\begin{lemma}[support reduction]\label{lem:support-reduction}
A graph $G$ is an exact $5$-leaf power if and only if it has a support-tree distance-three representation.
\end{lemma}

\begin{proof}
This is the case $k=5$ of \Cref{thm:support-model}.
\end{proof}

\begin{corollary}[induced-subgraph closure]\label{cor:induced-closure}
For $k\ge2$, every induced subgraph of an exact $k$-leaf power is an exact $k$-leaf power.
\end{corollary}

\begin{proof}
Restrict a support-tree distance-$(k-2)$ representation of $G$ to the chosen vertex set.  The same tree and the restricted support map represent the induced subgraph.
\end{proof}

A connected nontrivial exact odd-leaf power is bipartite: the supports in a tree have a bipartition, and odd support distances go across this bipartition.  In particular, exact $5$-leaf powers are bipartite.  False twins may share the same support.  In most structural statements below we work with connected false-twin-free bipartite graphs; false-twin stable modules can be expanded afterwards by attaching several leaves to the same support vertex.

\begin{definition}[stable false-twin reduction]
Two distinct vertices $u,v$ of a graph are \emph{false twins} if they are nonadjacent and
\[
 N_H(u)=N_H(v).
\]
A graph is \emph{false-twin-free} if no such pair exists.  The \emph{false-twin reduction} of $H$ is obtained by contracting each maximal stable false-twin class to one representative.
\end{definition}

\begin{proposition}[component and false-twin reduction]\label{prop:reduction}
A graph $G$ is an exact $5$-leaf power if and only if each connected component of its false-twin reduction is an exact $5$-leaf power.  Moreover, any exact $5$-leaf root of the reduced graph can be expanded to a root of $G$ by attaching all vertices in a false-twin class to the same support vertex.
\end{proposition}

\begin{proof}
Let $T$ be an exact $5$-leaf root and let $R$ be its support tree.  If two graph leaves have the same support vertex in $R$, then their distance in $T$ is $2$, so they are nonadjacent; and their distances to every other graph leaf are equal, so they have the same open neighbourhood.  Thus vertices sharing a support form a stable false-twin class.  Collapsing such a class to one representative preserves all distances from that class to the rest of the graph, and hence preserves exact $5$-representability.

Conversely, suppose a reduced graph has an exact $5$-leaf root.  For each false-twin class of the original graph, attach all leaves of that class to the support vertex of its representative.  Leaves attached to the same support are mutually at distance $2$, hence nonadjacent, and they have identical distances to all other leaves.  Thus the expanded tree represents the original graph.

Finally, disjoint unions are handled componentwise.  If several components have support trees, join these support trees by paths long enough that the support distance between vertices in different components is never $3$.  Then no cross-component leaf distance is $5$.  The reverse direction is immediate because induced subtrees on connected components give representations of the components.
\end{proof}

\section{Ladders and square-chain subgraphs}\label{sec:ladders}

The standard ladder family is the smallest place where overlapping induced squares create a genuine obstruction.  The sharp threshold is illustrated in \Cref{fig:ladder-threshold}.  We give a direct support-tree proof of its exact-$5$ classification.  The negative direction uses only the uniqueness of the exact $5$-leaf root of the domino.

\begin{definition}[$t$-square ladders]\label{def:ladder}
For $t\ge1$, let $L_t$ be the bipartite graph with parts
\[
 X_t=\{x_0,x_1,\ldots,x_t\},\qquad
 Y_t=\{y_0,y_1,\ldots,y_t\},
\]
and with
\[
 x_i y_j\in E(L_t)
 \quad\Longleftrightarrow\quad |i-j|\le1.
\]
Equivalently, $L_t$ is the usual ladder $P_{t+1}\square K_2$, written in the bipartite labelling in which the $i$th square is
\[
 Q_i=x_{i-1}-y_{i-1}-x_i-y_i-x_{i-1},
 \qquad i=1,\ldots,t.
\]
Thus $L_1=C_4$ and $L_2$ is the domino.

\end{definition}

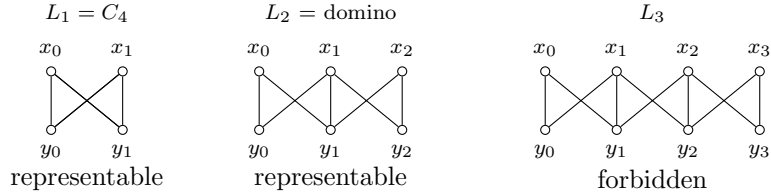
\begin{figure}[H]
\centering
\begin{tikzpicture}[scale=.86, every node/.style={font=\scriptsize}, v/.style={circle,draw,inner sep=1.2pt,fill=white}]
\begin{scope}[shift={(0,0)}]
\node at (.55,1.35) {$L_1=C_4$};
\node[v,label=above:$x_0$] (a0) at (0,.45) {};
\node[v,label=above:$x_1$] (a1) at (1.1,.45) {};
\node[v,label=below:$y_0$] (b0) at (0,-.45) {};
\node[v,label=below:$y_1$] (b1) at (1.1,-.45) {};
\draw (a0)--(b0)--(a1)--(b1)--(a0);
\draw (a0)--(b1) (a1)--(b0);
\node[font=\small] at (.55,-1.2) {representable};
\end{scope}
\begin{scope}[shift={(3.2,0)}]
\node at (1.1,1.35) {$L_2$ = domino};
\foreach \i in {0,1,2}{
  \node[v,label=above:$x_\i$] (x\i) at (\i*1.1,.45) {};
  \node[v,label=below:$y_\i$] (y\i) at (\i*1.1,-.45) {};
  \draw (x\i)--(y\i);
}
\draw (x0)--(y1) (x1)--(y0) (x1)--(y2) (x2)--(y1);
\node[font=\small] at (1.1,-1.2) {representable};
\end{scope}
\begin{scope}[shift={(7.6,0)}]
\node at (1.65,1.35) {$L_3$};
\foreach \i in {0,1,2,3}{
  \node[v,label=above:$x_\i$] (u\i) at (\i*1.1,.45) {};
  \node[v,label=below:$y_\i$] (v\i) at (\i*1.1,-.45) {};
  \draw (u\i)--(v\i);
}
\draw (u0)--(v1) (u1)--(v0) (u1)--(v2) (u2)--(v1) (u2)--(v3) (u3)--(v2);
\node[font=\small] at (1.65,-1.2) {forbidden};
\end{scope}
\end{tikzpicture}
\caption{The sharp ladder threshold.  The first square $L_1$ and the domino $L_2$ have exact $5$-leaf roots, but the three-square ladder $L_3$ is already forbidden.  Hence no longer ladder is representable, since it contains an induced $L_3$.}
\label{fig:ladder-threshold}
\end{figure}

We shall use the following support form of the unique domino root.  It is just the exact-$5$ domino root of Brandstaedt, Le and Rautenbach \cite[Lemma~6(iii)]{BLR2010}, after deleting the six graph leaves.  We state it with labels because the overlap argument for ladders depends on the boundary vertices.

\begin{lemma}[domino support form]\label{lem:domino-support-form}
Let $D$ be the domino with bipartition
\[
 A=\{a_0,a_1,a_2\},\qquad B=\{b_0,b_1,b_2\},
\]
and edges $a_i b_j$ exactly when $|i-j|\le1$.  In every support-tree distance-three representation of $D$, exactly one of the following two symmetric alternatives holds, up to reversing the order $0,1,2$.
\begin{enumerate}[label=\textup{(\alph*)}]
\item There is a vertex $c$ adjacent to $s_{a_0},s_{a_1},s_{a_2}$, a vertex $d$ adjacent to $c$ and to $s_{b_1}$, and the two boundary supports satisfy
\[
 s_{b_0}s_{a_2}\in E(R),\qquad s_{b_2}s_{a_0}\in E(R).
\]
We call this the $A$-oriented support form.
\item The symmetric $B$-oriented form holds, with the roles of $A$ and $B$ interchanged.
\end{enumerate}
\end{lemma}

\begin{proof}
Take an exact $5$-leaf root of $D$ and delete its graph leaves.  By \cite[Lemma~6(iii)]{BLR2010}, the domino has a unique exact $5$-leaf root, up to graph automorphisms.  Its support tree is precisely the tree described in (a), together with its images under the two symmetries of the labelled domino: reversal of the order $0,1,2$ and interchange of the two bipartition classes.  Conversely, the tree in (a) represents the domino because
\[
 \dist_R(s_{a_i},s_{b_j})=3
 \quad\Longleftrightarrow\quad |i-j|\le1,
\]
and no two supports from the same side are at distance three.  This proves the stated labelled alternatives.
\end{proof}

\begin{theorem}[ladders]\label{thm:ladder}
For $t\ge1$, the ladder $L_t$ is an exact $5$-leaf power if and only if
\[
 t\le2.
\]
\end{theorem}

\begin{proof}
By Lemma~\ref{lem:support-reduction}, it is enough for the positive cases to give support-tree distance-three representations.

For $L_1=K_{2,2}$, take a path $a b$ and attach the two $X$-support vertices $s_{x_0},s_{x_1}$ to $a$ and the two $Y$-support vertices $s_{y_0},s_{y_1}$ to $b$.  Then every $s_{x_i}$ is at distance three from every $s_{y_j}$, while the same-side distances are two.  Hence the exact distance-three graph is $K_{2,2}=L_1$.

For $L_2$, take a support tree with vertices
\[
 s_{x_0},s_{x_1},s_{x_2},\quad s_{y_0},s_{y_1},s_{y_2},\quad c,d
\]
and edges
\[
 cs_{x_0},\ cs_{x_1},\ cs_{x_2},\ cd,\ ds_{y_1},\ s_{x_2}s_{y_0},\ s_{x_0}s_{y_2}.
\]
Then
\[
 \dist(s_{x_i},s_{y_j})=3
 \quad\Longleftrightarrow\quad |i-j|\le1,
\]
for $i,j\in\{0,1,2\}$, and no same-side pair is at distance three.  Therefore the represented graph is $L_2$.

It remains to exclude $t\ge3$.  By Corollary~\ref{cor:induced-closure}, it is enough to exclude $L_3$.  Suppose, for a contradiction, that $L_3$ has a support-tree distance-three representation $R$.

Let $D_L$ be the induced domino on
\[
 \{x_0,x_1,x_2,y_0,y_1,y_2\},
\]
and let $D_R$ be the induced domino on
\[
 \{x_1,x_2,x_3,y_1,y_2,y_3\}.
\]
The restriction of $R$ to the supports of $D_L$ is a support representation of a domino, and similarly for $D_R$.  Hence Lemma~\ref{lem:domino-support-form} applies to both.

Assume first that $D_L$ is $X$-oriented.  Thus $s_{x_0},s_{x_1},s_{x_2}$ have a common neighbour $c$, while the boundary support $s_{y_0}$ is adjacent to $s_{x_2}$ and the boundary support $s_{y_2}$ is adjacent to $s_{x_0}$, up to reversing the left domino.  The two shared $Y$-supports $s_{y_1}$ and $s_{y_2}$ are then at distance four in the left domino support form.  Hence $D_R$ cannot be $Y$-oriented: in a $Y$-oriented support form the shared supports $s_{y_1}$ and $s_{y_2}$ would have a common neighbour and hence distance two, contradicting uniqueness of the path between them in the tree.  Thus $D_R$ is also $X$-oriented.

In the $X$-oriented form of $D_R$, the supports $s_{x_1}$ and $s_{x_2}$ have a common neighbour.  In a tree, two vertices at distance two have a unique common neighbour; since their common neighbour in $D_L$ is $c$, the common neighbour in $D_R$ is again $c$.  Therefore $s_{x_3}$ is adjacent to $c$.  But $s_{y_0}$ is adjacent to $s_{x_2}$ and $s_{x_2}$ is adjacent to $c$, so
\[
 \dist_R(s_{x_3},s_{y_0})=3.
\]
This would make $x_3y_0$ an edge of the represented graph, whereas $|3-0|>1$ and hence $x_3y_0\notin E(L_3)$.  This contradiction excludes the case that $D_L$ is $X$-oriented.

The case in which $D_L$ is $Y$-oriented is symmetric.  Then $D_R$ is forced to be $Y$-oriented, the shared neighbour of $s_{y_1}$ and $s_{y_2}$ is the same in both dominoes, and the boundary path forces
\[
 \dist_R(s_{x_0},s_{y_3})=3,
\]
contradicting the nonedge $x_0y_3$.  Thus $L_3$ is not an exact $5$-leaf power.  Since $L_t$ contains an induced $L_3$ for every $t\ge3$, no such $L_t$ is an exact $5$-leaf power.
\end{proof}

\begin{definition}[induced square-chains]\label{def:induced-square-chain}
An \emph{induced square-chain of length $t$} in a graph $G$ is an induced subgraph of $G$ isomorphic to the ladder $L_t$.  Equivalently, it is a sequence of $t$ induced squares in which consecutive squares share exactly one edge and the union has no additional edges.
\end{definition}

\begin{corollary}[induced square-chain obstruction]\label{cor:square-chain-obstruction}
No exact $5$-leaf power contains an induced square-chain of length three.  Equivalently, every induced square-chain in an exact $5$-leaf power has length at most two.
\end{corollary}

\begin{proof}
If an exact $5$-leaf power $G$ contained an induced square-chain of length three, then $G$ would contain an induced subgraph isomorphic to $L_3$.  By Corollary~\ref{cor:induced-closure}, this induced subgraph would itself be an exact $5$-leaf power, contradicting \Cref{thm:ladder}.  An induced square-chain of length $t\ge3$ contains an induced subchain of length three, so the equivalent formulation follows.
\end{proof}

\begin{remark}
The proof above isolates the local reason for the ladder obstruction.  One domino has a unique support form, but two overlapping dominoes force the same orientation through their common square.  The forced boundary attachment of the second domino then creates an exact-distance edge across the outer boundary of the three-square ladder.
\end{remark}

\section{Complete-bipartite block graphs}\label{sec:block-complete-bipartite}

The ladder theorem excludes long sparse square chains.  Dense square blocks behave differently.  Complete bipartite blocks may contain many induced squares, but all of those squares share the same two sides and can be represented by one support edge.  This gives a second complete subclass theorem, complementary to the ladder obstruction.

\begin{definition}[block-complete multipartite graphs]
A connected graph $G$ is \emph{block-complete multipartite} if every block of $G$ is a complete multipartite graph.  It is \emph{block-complete bipartite} if every block is a complete bipartite graph, with $K_2=K_{1,1}$ allowed.
\end{definition}

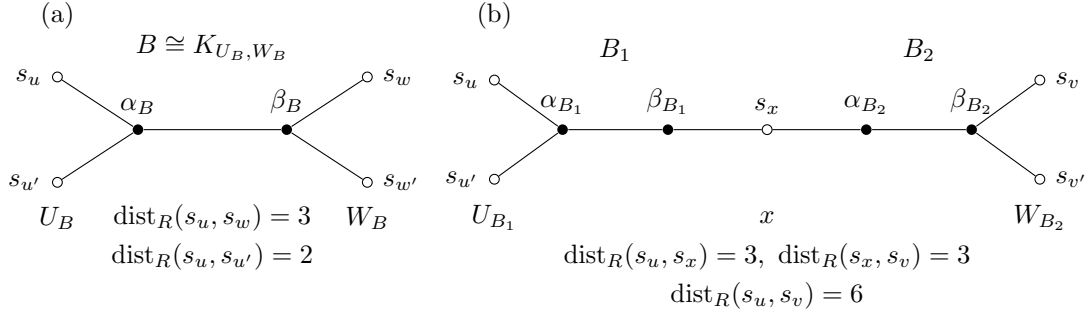
\begin{figure}[H]
\centering
\begin{tikzpicture}[scale=.82, every node/.style={font=\small}, sup/.style={circle,draw,inner sep=1.3pt}, hub/.style={circle,fill=black,inner sep=1.4pt}]
\node at (-5.05,1.85) {(a)};
\node[hub,label=above:$\alpha_B$] (a) at (-3.75,0) {};
\node[hub,label=above:$\beta_B$] (b) at (-1.35,0) {};
\draw (a)--(b);
\node[sup,label=left:$s_u$] (u1) at (-5.05,.85) {};
\node[sup,label=left:$s_{u'}$] (u2) at (-5.05,-.85) {};
\node[sup,label=right:$s_w$] (w1) at (-.05,.85) {};
\node[sup,label=right:$s_{w'}$] (w2) at (-.05,-.85) {};
\draw (u1)--(a)--(u2);
\draw (w1)--(b)--(w2);
\node at (-2.55,1.32) {$B\cong K_{U_B,W_B}$};
\node at (-5.05,-1.42) {$U_B$};
\node at (-.05,-1.42) {$W_B$};
\node at (-2.55,-1.42) {$\dist_R(s_u,s_w)=3$};
\node at (-2.55,-2.05) {$\dist_R(s_u,s_{u'})=2$};

\node at (2.0,1.85) {(b)};
\node[hub,label=above:$\alpha_{B_1}$] (a1) at (3.1,0) {};
\node[hub,label=above:$\beta_{B_1}$] (b1) at (4.8,0) {};
\node[sup,label=above:$s_x$] (sx) at (6.4,0) {};
\node[hub,label=above:$\alpha_{B_2}$] (a2) at (8.0,0) {};
\node[hub,label=above:$\beta_{B_2}$] (b2) at (9.7,0) {};
\draw (a1)--(b1)--(sx)--(a2)--(b2);
\node[sup,label=left:$s_u$] (su) at (2.0,.8) {};
\node[sup,label=left:$s_{u'}$] (suu) at (2.0,-.8) {};
\draw (su)--(a1)--(suu);
\node[sup,label=right:$s_v$] (sv) at (10.8,.8) {};
\node[sup,label=right:$s_{v'}$] (svv) at (10.8,-.8) {};
\draw (sv)--(b2)--(svv);
\node at (3.95,1.28) {$B_1$};
\node at (8.85,1.28) {$B_2$};
\node at (2.0,-1.42) {$U_{B_1}$};
\node at (6.4,-1.42) {$x$};
\node at (10.8,-1.42) {$W_{B_2}$};
\node at (6.4,-2.05) {$\dist_R(s_u,s_x)=3,\ \dist_R(s_x,s_v)=3$};
\node at (6.4,-2.68) {$\dist_R(s_u,s_v)=6$};
\end{tikzpicture}
\caption{Reader's guide to the exact-$5$ construction for block-complete bipartite graphs.  In panel (a), one complete bipartite block $B\cong K_{U_B,W_B}$ is represented by a single support edge $\alpha_B\beta_B$: opposite-side vertices have support distance $3$, while same-side vertices have support distance $2$.  In panel (b), two block gadgets are glued by identifying the support $s_x$ of the common cutvertex $x$.  Distances inside each block are preserved, while vertices in different blocks do not create new exact-distance-$3$ pairs; for example, the displayed supports satisfy $\dist_R(s_u,s_v)=6$.}
\label{fig:block-bipartite-gadget}
\end{figure}

\begin{theorem}[block-complete multipartite graphs]\label{thm:block-complete-multipartite}
Let $G$ be a connected graph whose blocks are complete multipartite graphs.  Then the following are equivalent.
\begin{enumerate}[label=\textup{(\roman*)}]
\item $G$ is an exact $5$-leaf power.
\item $G$ is bipartite.
\item Every block of $G$ is complete bipartite.
\end{enumerate}
Moreover, if these conditions hold, then $G$ has an explicit support-tree distance-three representation obtained by replacing every block $B\cong K_{U_B,W_B}$ by one support edge $\alpha_B\beta_B$.
\end{theorem}

\begin{proof}
Every exact $5$-leaf power is bipartite.  Indeed, by \Cref{lem:support-reduction}, it has a support-tree distance-three representation on a tree $R$.  A bipartition of $R$ induces a bipartition of the represented graph, because all represented edges join support vertices at odd distance three.  Thus (i) implies (ii).

If $G$ is bipartite and a block $B$ of $G$ is complete multipartite, then $B$ has at most two parts: three nonempty parts would contain a triangle.  Since a nontrivial block has no isolated part by itself, $B$ is complete bipartite.  Hence (ii) implies (iii).  The implication (iii) implies (ii) is immediate because the block graph of a connected graph is a tree and the bipartitions of complete bipartite blocks agree across cutvertices.

It remains to prove (iii) implies (i).  Fix a global bipartition $V(G)=X\cup Y$.  For each block $B$ write
\[
 U_B=B\cap X,\qquad W_B=B\cap Y,
\]
so that $B\cong K_{U_B,W_B}$.  We construct a support tree $R$.

For every graph vertex $v\in V(G)$ introduce a support vertex $s_v$.  For every block $B$ introduce two support vertices $\alpha_B$ and $\beta_B$, and add the support edge
\[
 \alpha_B\beta_B.
\]
For each incidence $v\in B$, add
\[
 s_v\alpha_B \quad\text{if } v\in U_B,
 \qquad
 s_v\beta_B \quad\text{if } v\in W_B.
\]

The graph $R$ is a tree.  To see this, start from the usual block-vertex incidence graph of $G$, whose vertices are the graph vertices and the blocks, and where $v$ is adjacent to $B$ exactly when $v\in B$.  This incidence graph is a tree for connected $G$ after non-cutvertices are included as leaves.  The graph $R$ is obtained from it by replacing each block vertex $B$ by the edge $\alpha_B\beta_B$ and attaching the incidences of vertices in $U_B$ to $\alpha_B$ and those of vertices in $W_B$ to $\beta_B$.  Replacing a vertex of a tree by a tree preserves connectedness and acyclicity.

Now let $u,v$ be distinct vertices of $G$.  If $u$ and $v$ lie in a common block $B$ and are in opposite bipartition classes, then the path in $R$ is
\[
 s_u-\alpha_B-\beta_B-s_v
\]
up to interchanging $\alpha_B$ and $\beta_B$, so
\[
 \dist_R(s_u,s_v)=3.
\]
If they lie in the same part of the same block, then their supports have distance two through $\alpha_B$ or through $\beta_B$.  Finally, if $u$ and $v$ do not lie in a common block, then the path between $s_u$ and $s_v$ in the block-vertex incidence tree uses at least two block vertices and one separating cutvertex; after the above replacement its length in $R$ is at least four.  Hence
\[
 \dist_R(s_u,s_v)=3
 \quad\Longleftrightarrow\quad
 u\text{ and }v\text{ are in opposite parts of a common block}
 \quad\Longleftrightarrow\quad
 uv\in E(G).
\]
Thus $R$ is a support-tree distance-three representation of $G$, and \Cref{lem:support-reduction} gives an exact $5$-leaf root.
\end{proof}

\begin{corollary}[complete bipartite graphs and trees]\label{cor:complete-bipartite-trees}
Every complete bipartite graph and every tree is an exact $5$-leaf power.  More generally, every connected graph whose blocks are complete bipartite graphs is an exact $5$-leaf power.
\end{corollary}

\begin{proof}
Complete bipartite graphs have one complete bipartite block, and trees have only $K_2$ blocks.  The final assertion is exactly the constructive direction of \Cref{thm:block-complete-multipartite}.
\end{proof}

\begin{corollary}[cacti]\label{cor:cacti}
Let $G$ be a cactus.  Then $G$ is an exact $5$-leaf power if and only if $G$ is bipartite, equivalently if and only if every cycle block of $G$ is even.
\end{corollary}

\begin{proof}
The necessity is bipartiteness of exact odd-leaf powers, as above.  Conversely, Brandstaedt, Le and Rautenbach proved that all bipartite cacti are exact $5$-leaf powers \cite[Corollary~4]{BLR2010}.  Equivalently, one may build the root by the standard block induction using $K_2$ blocks and exact $5$-leaf roots of even chordless cycle blocks.
\end{proof}

\begin{corollary}[maximum-degree-two graphs]\label{cor:max-degree-two}
Let $G$ be a graph with maximum degree at most two.  Then $G$ is an exact $5$-leaf power if and only if $G$ is bipartite.
\end{corollary}

\begin{proof}
Each connected component of $G$ is a path or a cycle.  Paths are trees and hence are exact $5$-leaf powers by \Cref{cor:complete-bipartite-trees}.  Even cycles are covered by the cycle block-language theorem for length at least eight, together with the finite cases $C_4$ and $C_6$ of \cite[Lemma~6]{BLR2010}.  Odd cycles are excluded by bipartiteness of exact odd-leaf powers.  Components are represented independently by \Cref{prop:reduction}.
\end{proof}

\begin{remark}[dense squares versus square chains]
\Cref{thm:ladder,thm:block-complete-multipartite} should be read together.  They show that the obstruction is not the mere presence of many induced squares.  A complete bipartite block may contain arbitrarily many induced $C_4$'s, but they are all carried by one support edge $\alpha_B\beta_B$.  By contrast, $L_3$ is bipartite and sparse, yet the third square in its induced square-chain already forces an obstruction.  Thus exact distance five distinguishes dense square blocks from square propagation along a chain.
\end{remark}

\section{Bipartite co-cluster graphs and crowns}\label{sec:cocluster}

The complete-bipartite block theorem is a positive result for dense blocks.  Another dense bipartite family is obtained by starting from a complete bipartite graph and deleting a structured disjoint union of complete bipartite subgraphs in the bipartite complement.  The resulting class contains all crown graphs and, in particular, the cube $Q_3$; see \Cref{fig:crowns-coclusters}.

\begin{definition}[bipartite co-cluster graphs]
Let $G=(X,Y;E)$ be a bipartite graph.  We call $G$ a \emph{bipartite co-cluster graph} if there are pairwise disjoint sets
\[
 X=X_0\dot\cup X_1\dot\cup\cdots\dot\cup X_t,
 \qquad
 Y=Y_0\dot\cup Y_1\dot\cup\cdots\dot\cup Y_t,
\]
where $X_i$ and $Y_i$ are nonempty for $i\ge1$, such that
\[
 xy\notin E(G)
 \quad\Longleftrightarrow\quad
 x\in X_i,\ y\in Y_i\text{ for some }i\in\{1,\ldots,t\}.
\]
Equivalently, the bipartite complement of $G$ is a disjoint union of complete bipartite graphs $K_{X_i,Y_i}$, with possible isolated vertices $X_0\cup Y_0$.

\end{definition}

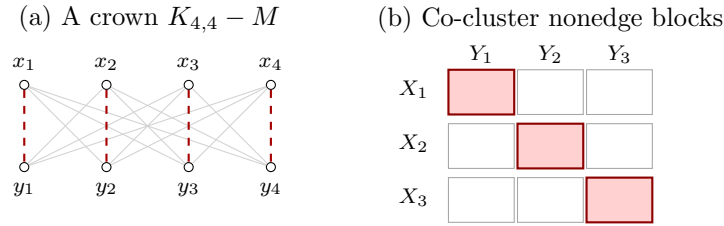
\begin{figure}[H]
\centering
\begin{tikzpicture}[scale=.87, every node/.style={font=\scriptsize}, v/.style={circle,draw,inner sep=1.2pt,fill=white}]
\begin{scope}[shift={(0,0)}]
\node[font=\small] at (1.9,2.25) {(a) A crown $K_{4,4}-M$};
\foreach \i in {1,2,3,4}{
  \node[v,label=above:$x_\i$] (x\i) at ({(\i-1)*1.25},1.25) {};
  \node[v,label=below:$y_\i$] (y\i) at ({(\i-1)*1.25},0) {};
}
\foreach \i in {1,2,3,4}{
  \foreach \j in {1,2,3,4}{
    \ifnum\i=\j\relax\else\draw[gray!35] (x\i)--(y\j);\fi
  }
}
\foreach \i in {1,2,3,4}{\draw[red!65!black,dashed,thick] (x\i)--(y\i);}
\end{scope}
\begin{scope}[shift={(6.2,0)}]
\node[font=\small] at (1.8,2.25) {(b) Co-cluster nonedge blocks};
\foreach \i/\lab in {0/$Y_1$,1/$Y_2$,2/$Y_3$}{\node at ({.75+\i*1.05},1.72) {\lab};}
\foreach \i/\lab in {0/$X_1$,1/$X_2$,2/$X_3$}{\node at (-.28,{1.18-\i*.82}) {\lab};}
\foreach \r in {0,1,2}{
  \foreach \c in {0,1,2}{
    \pgfmathsetmacro{\xx}{.25+\c*1.05}
    \pgfmathsetmacro{\yy}{.8-\r*.82}
    \ifnum\r=\c
      \fill[red!18] (\xx,\yy) rectangle +(1,.68);
      \draw[red!55!black,thick] (\xx,\yy) rectangle +(1,.68);
    \else
      \draw[gray!70] (\xx,\yy) rectangle +(1,.68);
    \fi
  }
}
\end{scope}
\end{tikzpicture}
\caption{Dense bipartite graphs with structured missing edges.  In the crown graph the dashed pairs are the deleted matching.  In the co-cluster pattern, the shaded cells are the complete bipartite nonedge blocks $X_i\times Y_i$ in the bipartite complement.}
\label{fig:crowns-coclusters}
\end{figure}

\begin{theorem}[bipartite co-cluster graphs]\label{thm:cocluster}
Every bipartite co-cluster graph is an exact $k$-leaf power for every $k\ge5$.
\end{theorem}

\begin{proof}
Let $G=(X,Y;E)$ have the decomposition in the definition, and put
\[
 d=k-2\ge3.
\]
By \Cref{thm:support-model}, it is enough to construct a support tree $R$ in which adjacency is exactly support distance $d$.

Take a centre $c$.  For each $i\in\{1,\ldots,t\}$ take a path
\[
 c=p_{i,0}p_{i,1}\cdots p_{i,d-1}
\]
of length $d-1$ from $c$.  Place all vertices of $X_i$ at the support $p_{i,1}$, and all vertices of $Y_i$ at the support $p_{i,d-1}$.  If $X_0$ is nonempty, add one further branch $c=x_0^*$ of length one and place all vertices of $X_0$ at $x_0^*$.  If $Y_0$ is nonempty, add another branch
\[
 c=y_{0,0}y_{0,1}\cdots y_{0,d-1}
\]
of length $d-1$ and place all vertices of $Y_0$ at $y_{0,d-1}$.

We check distances.  If $x\in X_i$ and $y\in Y_i$ for some $i\ge1$, then their supports lie on the same branch and
\[
 \dist_R(s_x,s_y)=d-2\ne d,
\]
so the prescribed nonedges are not represented.  If $x\in X_i$ and $y\in Y_j$ with $i\ne j$, where $i,j\ge1$, then the path goes through $c$ and
\[
 \dist_R(s_x,s_y)=1+(d-1)=d.
\]
The same formula holds for pairs with $x\in X_0$ and $y\in Y_j$, or with $x\in X_i$ and $y\in Y_0$, and also for $x\in X_0$, $y\in Y_0$ when both sets are present, because the corresponding supports lie on different branches at distances $1$ and $d-1$ from $c$.

Finally, pairs on the same side are never represented.  Two $X$-supports have distance $0$ or $2$, and two $Y$-supports have distance $0$ or $2d-2$; none of these values equals $d$ because $d\ge3$.  Hence support distance $d$ gives exactly the edges of $G$, and attaching one leaf to each support gives an exact $k$-leaf root.
\end{proof}

\begin{corollary}[crowns and the cube]\label{cor:crowns-cube}
For every $n\ge2$, the crown graph $K_{n,n}-M$ is an exact $k$-leaf power for every $k\ge5$.  In particular, the cube $Q_3\cong K_{4,4}-M$ is an exact $5$-leaf power.
\end{corollary}

\begin{proof}
A crown graph is the bipartite co-cluster graph obtained by taking $|X_i|=|Y_i|=1$ for $i=1,\ldots,n$ and $X_0=Y_0=\varnothing$.  The usual bipartition of $Q_3$ into even and odd vertices gives $Q_3\cong K_{4,4}$ minus the perfect matching joining each even vertex to its antipode.
\end{proof}

\begin{remark}
The co-cluster theorem is independent of the block-complete bipartite theorem.  For $n\ge3$, the crown graph $K_{n,n}-M$ is two-connected and is not complete bipartite, but it is still represented by the single-centre support construction above.  Thus exact distance five admits both complete bipartite blocks and dense bipartite graphs with structured missing bicliques.
\end{remark}

\section{Complete multipartite blocks at larger exact distances}\label{sec:larger-k-blocks}

The exact $5$ theorem for block-complete multipartite graphs is the first odd member of a broader pattern.  In a single block the classification is especially clean for every exact distance.  For block decompositions, all odd distances $k\ge5$ behave like distance five, while even distances $k\ge8$ allow arbitrary complete multipartite blocks.  The exceptional even boundary $k=6$ is not included in the block theorem below because the same construction can create new exact-$6$ pairs across adjacent blocks.

\begin{lemma}[distance lifting]\label{lem:distance-lifting}
If a graph $G$ is an exact $k$-leaf power, then $G$ is an exact $(k+2)$-leaf power.
\end{lemma}

\begin{proof}
Let $T$ be an exact $k$-leaf root of $G$.  Replace every edge incident with a graph leaf by a path of length two.  The graph leaves remain exactly the original leaves, and the distance between any two graph leaves increases by two.  Thus the same graph is represented at exact distance $k+2$.
\end{proof}

\begin{theorem}[complete multipartite graphs]\label{thm:complete-multipartite-all-k}
Let $k\ge3$, and let
\[
 G=K_{n_1,\ldots,n_s}
\]
be a complete multipartite graph with nonempty partite classes.  Then $G$ is an exact $k$-leaf power if and only if either $k$ is even or $s\le2$.
\end{theorem}

\begin{proof}
If $k$ is odd, then every exact $k$-leaf power is bipartite.  Indeed, a bipartition of an exact $k$-leaf root separates all pairs of leaves at odd distance.  A complete multipartite graph with at least three nonempty parts contains a triangle.  Hence, for odd $k$, exact $k$-leaf representability forces $s\le2$.

It remains to give the constructions.  First let $k=2q+1$ be odd, with $q\ge1$, and assume $s\le2$.  If $s=1$, then $G$ is edgeless.  The one-vertex case is represented by the one-vertex tree; otherwise a star with the vertices of $G$ as leaves has all leaf distances equal to $2\ne k$.  If $s=2$, write the two parts as $X$ and $Y$.  Take an edge $ab$.  For every $x\in X$, attach a path of length $q$ from $a$ to the leaf $x$; for every $y\in Y$, attach a path of length $q$ from $b$ to the leaf $y$.  Then two leaves in opposite parts have distance
\[
 q+1+q=2q+1=k,
\]
whereas two leaves in the same part have distance $2q=k-1$.  Thus the exact $k$-leaf graph is $K_{X,Y}$.

Now let $k=2q$ be even, with $q\ge2$.  If $s=1$, use the same edgeless construction as above.  If $s\ge2$, take a centre $c$.  For each part $P_i$, take a vertex $h_i$ at distance $q-1$ from $c$, the paths from $c$ to the different $h_i$ being internally disjoint.  Attach every vertex of $P_i$ as a leaf adjacent to $h_i$.  Two leaves in the same part have distance $2$, while leaves in different parts have distance
\[
 1+(q-1)+(q-1)+1=2q=k.
\]
Hence the exact $k$-leaf graph is exactly $K_{n_1,\ldots,n_s}$.
\end{proof}

\begin{theorem}[block-complete multipartite graphs at larger distances]\label{thm:block-complete-larger-k}
Let $G$ be a connected graph whose blocks are complete multipartite graphs.
\begin{enumerate}[label=\textup{(\roman*)}]
\item If $k\ge5$ is odd, then $G$ is an exact $k$-leaf power if and only if $G$ is bipartite, equivalently if and only if every block of $G$ is complete bipartite.
\item If $k\ge8$ is even, then $G$ is an exact $k$-leaf power.
\end{enumerate}
\end{theorem}

\begin{proof}
Let $k=2q+1\ge5$ be odd, so $q\ge2$.  The necessity of bipartiteness is the same parity argument as above.  In a block-complete multipartite graph, bipartiteness is equivalent to every block being complete bipartite.

Assume now that $G$ is bipartite.  Fix its global bipartition $V(G)=X\cup Y$.  We construct a support tree $R$ in which adjacency is exactly support distance $k-2=2q-1$; attaching one new leaf to each support vertex then gives an exact $k$-leaf root.  For every graph vertex $v$, introduce a support vertex $s_v$.  For each block $B$, write
\[
 U_B=B\cap X,\qquad W_B=B\cap Y,
\]
and introduce two block hubs $\alpha_B,\beta_B$ joined by an edge.  For every incidence $v\in B$, add an internally disjoint path of length $q-1$ from $s_v$ to $\alpha_B$ if $v\in U_B$, and from $s_v$ to $\beta_B$ if $v\in W_B$.

This graph $R$ is a tree: it is obtained from the block-vertex incidence tree of $G$ by replacing each block vertex by the edge $\alpha_B\beta_B$ and by subdividing incidence edges.  If $u,v$ lie in opposite sides of a common block, then
\[
 \dist_R(s_u,s_v)=(q-1)+1+(q-1)=2q-1=k-2.
\]
If they lie in the same side of a common block, their support distance is $2q-2$.  If they do not lie in a common block, the path between their supports passes through at least two block incidences; its length is at least
\[
 2(2q-2)=4q-4>2q-1=k-2.
\]
Therefore the exact distance-$k$ leaf graph obtained from $R$ is precisely $G$.

Now let $k=2q\ge8$ be even, so $q\ge4$.  We again construct a support tree $R$ in which adjacency is exactly support distance $k-2=2q-2$.  For every vertex $v\in V(G)$ take a support vertex $s_v$.  For each block $B$, let
\[
 P^B_1,\ldots,P^B_{r_B}
\]
be its partite classes.  Introduce a block centre $c_B$ and part hubs $h^B_1,\ldots,h^B_{r_B}$, each adjacent to $c_B$.  For every incidence $v\in B$ with $v\in P^B_i$, add an internally disjoint path of length $q-2$ from $s_v$ to $h^B_i$.

Again $R$ is a tree, because it is obtained from the block-vertex incidence tree by replacing each block vertex by a star and subdividing incidence edges.  If $u,v$ lie in different parts of the same block, then
\[
 \dist_R(s_u,s_v)=(q-2)+2+(q-2)=2q-2=k-2.
\]
If they lie in the same part of a common block, their support distance is $2q-4$.  If they do not lie in a common block, their path contains at least two block incidences, and hence has support length at least
\[
 2(2q-4)=4q-8>2q-2=k-2,
\]
because $q\ge4$.  Thus attaching one leaf to every $s_v$ gives an exact $k$-leaf root of $G$.
\end{proof}

\begin{corollary}[larger-distance subclass consequences]\label{cor:larger-distance-subclasses}
The following hold.
\begin{enumerate}[label=\textup{(\roman*)}]
\item Every complete bipartite graph is an exact $k$-leaf power for every $k\ge3$.
\item Every complete multipartite graph is an exact $k$-leaf power for every even $k\ge4$.
\item Every block-complete bipartite graph is an exact $k$-leaf power for every odd $k\ge5$ and every even $k\ge8$.
\item A connected block graph is an exact $k$-leaf power, for $k\ge3$, if and only if $k$ is even or the block graph is a tree.
\end{enumerate}
\end{corollary}

\begin{proof}
Parts (i) and (ii) are immediate from \Cref{thm:complete-multipartite-all-k}.  Part (iii) follows from \Cref{thm:block-complete-larger-k}.  For (iv), every block graph is an exact $4$-leaf power by the exact-$4$ characterization of Brandstaedt, Le and Rautenbach \cite[Theorem~2]{BLR2010}, and \Cref{lem:distance-lifting} gives all even $k\ge4$.  If $k$ is odd, then every exact $k$-leaf power is bipartite; a bipartite block graph has only $K_2$ blocks, hence is a tree.  Conversely, a tree has an exact $3$-leaf root obtained by attaching one new leaf to every vertex of the tree, and \Cref{lem:distance-lifting} gives all odd $k\ge3$.
\end{proof}

\subsection{The exact-six boundary and rigid short shadows}\label{subsec:k6-boundary}

The even theorem above deliberately starts at $k=8$.  At $k=6$ the support distance is four, and two distance-two same-part shadows through adjacent blocks may collide and create a new edge.  We isolate the local mechanism here.  The material in this subsection is not used in the exact-$5$ structure theorems; it explains why the boundary case is genuinely different and records the first unavoidable obstruction for any complete $k=6$ classification of block-complete multipartite graphs.

Let $T$ be an exact $6$-leaf root of a graph $G$, and delete the graph leaves from $T$.  If $s_v$ denotes the neighbour of the leaf $v$ in $T$, then the support tree $R$ satisfies
\[
 uv\in E(G) \quad\Longleftrightarrow\quad \dist_R(s_u,s_v)=4.
\]
Thus $k=6$ is the support-distance-four case.  A single complete multipartite block
\[
 B=K_{P_1,\ldots,P_r}
\]
is represented by a centre $c_B$, one part hub $h_i$ adjacent to $c_B$ for each part $P_i$, and support vertices for the vertices of $P_i$ adjacent to $h_i$.  Vertices in different parts are then at support distance four, while vertices in the same part are at support distance two.  In a block decomposition, the latter distance-two shadows are harmless inside one block but may be dangerous at cutvertices.

\begin{definition}[rigid incidence]
Let $G$ be a block-complete multipartite graph.  If $B$ is a block and $x$ is a cutvertex of $G$ contained in $B$, let $P_B(x)$ denote the partite class of $B$ containing $x$.  We call the incidence $(x,B)$ \emph{rigid} if $B$ has at least three partite classes and $|P_B(x)|\ge2$.
\end{definition}

The terminology comes from the following elementary tree-metric fact.  In a support-distance-four representation of a complete multipartite block with at least three partite classes, a vertex in the same part as a cutvertex is forced to sit at support distance two from that cutvertex, on the same side of the block centre.

\begin{lemma}[short shadows forced by rigid incidences]\label{lem:k6-rigid-shadow}
Let $G$ be a block-complete multipartite graph with a support tree $R$ in which adjacency is exactly support distance four.  Let $(x,B)$ be a rigid incidence, and let $a\in P_B(x)\setminus\{x\}$.  Then
\[
 \dist_R(s_a,s_x)=2.
\]
More precisely, if $u$ and $v$ are vertices in two partite classes of $B$ different from $P_B(x)$, and $m_B$ is the median of $s_x,s_u,s_v$, then $s_a$ lies at distance two from $m_B$ in the same component of $R-m_B$ as $s_x$.
\end{lemma}

\begin{proof}
The vertices $x,u,v$ lie in three distinct partite classes of $B$, so their supports are pairwise at distance four.  Hence their median $m_B$ satisfies
\[
 \dist_R(m_B,s_x)=\dist_R(m_B,s_u)=\dist_R(m_B,s_v)=2.
\]
Since $a$ is adjacent to $u$ and to $v$, $s_a$ is at distance four from both $s_u$ and $s_v$.  In a tree, the points at distance four from both $s_u$ and $s_v$ are precisely the points at distance two from $m_B$ outside the two components of $R-m_B$ containing $s_u$ and $s_v$.  Since $a$ and $x$ are in the same partite class of $B$, $s_a$ is not at distance four from $s_x$.  Therefore $s_a$ must lie in the same component of $R-m_B$ as $s_x$ and at distance two from $m_B$.  Because $x$ is a cutvertex, $x$ has a neighbour outside $B$.  If $s_a=s_x$, then $a$ would be at support distance four from every support-neighbour of $x$ in an incident block, creating cross-block edges.  Thus $s_a\ne s_x$, and the two vertices at distance two from $m_B$ in the same component are separated by distance two.
\end{proof}

The forced short-shadow phenomenon gives a useful sufficient condition.  We record it as the first formal boundary theorem for the exact-six case: it supplies a large positive family at $k=6$ and separates this boundary from the unresolved full compatibility problem.

\begin{definition}[admissible exact-six orientation]\label{def:k6-admissible}
Let $G$ be a block-complete multipartite graph.  For every complete bipartite block $B=K_{P,Q}$ choose one of $P,Q$ as the \emph{short side} of $B$.  For a block with at least three partite classes, every non-singleton part is declared short.  A cutvertex $x$ is \emph{exposed in an incident block $B$} if $P_B(x)$ is short and $|P_B(x)|\ge2$.  The orientation is \emph{admissible} if every cutvertex is exposed in at most one incident block.
\end{definition}

\begin{theorem}[admissible exact-six orientation is sufficient]\label{thm:k6-sufficient}
Let $G$ be a connected block-complete multipartite graph.  If $G$ has an admissible exact-six orientation, then $G$ is an exact $6$-leaf power.
\end{theorem}

\begin{proof}
We construct a support tree $R$ in which adjacency is exactly support distance four; adding one leaf adjacent to each support vertex gives an exact $6$-leaf root.

For each graph vertex $v$ introduce a support vertex $s_v$.  If $B$ is a block with at least three partite classes $P_1,\ldots,P_r$, introduce a centre $c_B$ and part hubs $h^B_1,\ldots,h^B_r$, each adjacent to $c_B$.  For every incidence $v\in P_i$, add the edge $s_vh^B_i$.

If $B=K_{P,Q}$ is bipartite and $P$ is the chosen short side, introduce a block hub $m_B$.  For $v\in P$, add the edge $s_vm_B$.  For $w\in Q$, add an internally disjoint path of length three from $s_w$ to $m_B$.  If $Q$ is the chosen short side, interchange the roles of $P$ and $Q$.

Gluing the gadgets along the common support vertices of cutvertices gives a tree, because the block-cut incidence graph of a connected graph is a tree and each block node has been replaced by a tree.  Inside a non-bipartite multipartite block, supports in different parts are at distance $1+2+1=4$, while supports in the same part have distance two.  Inside a bipartite block, a short-side support and a long-side support have distance $1+3=4$, while two short-side supports have distance two and two long-side supports have distance six.  Thus each block is represented correctly.

It remains to exclude cross-block distance-four pairs.  Let $u$ and $v$ lie in different blocks.  The path from $s_u$ to $s_v$ in $R$ passes through the support of at least one cutvertex $x$.  The only way for its total length to be four is that, at some such cutvertex $x$, the two incident block segments on the path both have length two.  A length-two segment from $s_x$ into a block exists exactly when $x$ is exposed in that block.  The admissibility condition forbids two such exposed incidences at the same cutvertex.  Hence no cross-block pair has support distance four.  Therefore $R$ represents exactly $G$ at support distance four.
\end{proof}

The next result is a complete boundary theorem for the first nontrivial local case: all blocks share one cutvertex.  It shows exactly how many rigid short shadows can be paired at a single centre; see \Cref{fig:k6-fan-boundary}.

\begin{definition}[multipartite block fan]
A \emph{multipartite block fan} is a connected block-complete multipartite graph with a distinguished cutvertex $x$ such that every block contains $x$ and any two distinct blocks meet exactly in $x$.  Thus $x$ is the only cutvertex.  Let
\[
 \rho_x(G)=|\{B: \text{$B$ has at least three partite classes and } |P_B(x)|\ge2\}|.
\]
The blocks counted by $\rho_x(G)$ are the rigid blocks at the centre of the fan.

\end{definition}

\begin{figure}[H]
\centering
\begin{tikzpicture}[>=Latex, every node/.style={font=\small}, box/.style={draw,rounded corners,align=center,minimum width=1.55cm,minimum height=.68cm,fill=gray!6}]
\node at (0,1.95) {(a) pairable boundary};
\node[circle,draw,inner sep=1.2pt,label=below:$x$] (xl) at (0,0) {};
\node[box] (l1) at (-1.25,.9) {$B_1$};
\node[box] (l2) at (1.25,.9) {$B_2$};
\draw (xl)--(l1) (xl)--(l2);
\node[font=\scriptsize,align=center] at (0,-.72) {$\rho_x=2$; two rigid shadows\\can be centre-paired};
\node[font=\small] at (0,-1.45) {exact $6$};

\begin{scope}[xshift=5.8cm]
\node at (0,1.95) {(b) fan obstruction};
\node[circle,draw,inner sep=1.2pt,label=below:$x$] (xr) at (0,0) {};
\node[box] (r1) at (-1.55,.9) {$B_1$};
\node[box] (r2) at (0,1.25) {$B_2$};
\node[box] (r3) at (1.55,.9) {$B_3$};
\draw (xr)--(r1) (xr)--(r2) (xr)--(r3);
\node[font=\scriptsize,align=center] at (0,-.72) {$\rho_x=3$; three rigid shadows\\force a cross-block distance four};
\node[font=\small] at (0,-1.45) {not exact $6$};
\end{scope}
\end{tikzpicture}
\caption{The exact-six fan boundary.  Each box denotes a rigid complete tripartite block sharing the cutvertex $x$.  In a one-cutvertex fan, at most two rigid blocks can be absorbed by centre-pairing; three already force an unintended support-distance-four pair.}
\label{fig:k6-fan-boundary}
\end{figure}
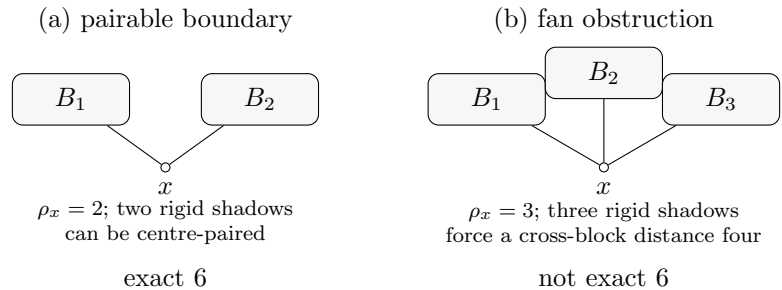

\begin{theorem}[the exact-six fan theorem]
\label{thm:k6-fan}
Let $G$ be a multipartite block fan with centre $x$.  Then
\[
 G \text{ is an exact }6\text{-leaf power}
 \quad\Longleftrightarrow\quad
 \rho_x(G)\le2 .
\]
\end{theorem}

\begin{proof}
We work in support-distance-four form.  First assume $\rho_x(G)\le2$ and construct a support tree $R$.

If $B=K_{P,Q}$ is bipartite and $x\in P$, make the $x$-side long.  Introduce a block hub $m_B$ and a path of length three from $s_x$ to $m_B$.  For every $v\in P\setminus\{x\}$ add an internally disjoint path of length three from $s_v$ to $m_B$, and for every $w\in Q$ add the edge $s_wm_B$.  Then opposite sides have support distance four, the $P$-side has support distance six, and the $Q$-side has support distance two.

If $B$ has at least three partite classes and $P_B(x)=\{x\}$, introduce a centre $c_B$ and a path $s_xq_Bc_B$.  For every other part $P$ of $B$, introduce a hub $h^B_P$ adjacent to $c_B$, and attach all supports $s_v$, $v\in P$, to $h^B_P$.  Then $x$ is at support distance four from every other vertex of $B$, and two non-central vertices are at support distance four exactly when they lie in different parts.

If $B$ is the only rigid block at $x$, use the same centre-path $s_xq_Bc_B$.  For the part $P_B(x)$, attach every support $s_a$ with $a\in P_B(x)\setminus\{x\}$ to $q_B$.  For all other parts use hubs adjacent to $c_B$ as above.  Vertices in $P_B(x)$ are at support distance two from $s_x$ and from each other, and they are at support distance four from every vertex in the other parts of $B$.

It remains to describe the case of two rigid blocks, say $B$ and $C$.  Introduce a node $q$ adjacent to $s_x$, and introduce two centres $c_B$ and $c_C$, both adjacent to $q$.  For every part of $B$ different from $P_B(x)$ attach a part hub to $c_B$, and attach the supports of vertices in that part to the hub.  For every vertex $a\in P_B(x)\setminus\{x\}$ use the support $s_a=c_C$.  Symmetrically, for every vertex $a\in P_C(x)\setminus\{x\}$ use the support $s_a=c_B$, and attach the other parts of $C$ to hubs adjacent to $c_C$.

All remaining non-rigid blocks are attached to $s_x$ by the first two constructions above.  The resulting graph $R$ is a tree: it is obtained by gluing block gadgets at the single support $s_x$, except in the paired rigid case where the two rigid centres share the node $q$.

We now verify distances.  Each bipartite block is represented correctly by the long/short construction.  A non-rigid multipartite block with $P_B(x)=\{x\}$ is represented correctly by its centre $c_B$.  A single unpaired rigid block is represented correctly by the path $s_xq_Bc_B$.  In the paired case, consider block $B$.  A vertex in $P_B(x)\setminus\{x\}$ has support $c_C$, hence it is at distance two from $s_x$ and at distance four from any support in a different part of $B$, because the path is
\[
 c_C-q-c_B-h^B_P-s_v .
\]
It is not at distance four from vertices of $C$: it is at distance two from the supports in the non-central parts of $C$, and at distance two from the supports in $P_C(x)\setminus\{x\}$.  The same argument holds with $B$ and $C$ interchanged.  Finally, every support in a non-rigid block different from $x$ is at distance at least four from $s_x$, while every rigid paired mate is at distance two from $s_x$ and every rigid non-mate support is at distance four from $s_x$.  Hence a path between supports belonging to different blocks has length different from four, except when the two vertices are both in a common block already handled above.  Thus $R$ represents exactly $G$ at support distance four, and $G$ is an exact $6$-leaf power.

Conversely, suppose that $G$ has a support tree $R$ representing adjacency by support distance four, and assume that $\rho_x(G)\ge3$.  Choose three rigid blocks $B_1,B_2,B_3$.  For each $i$, choose
\[
 a_i\in P_{B_i}(x)\setminus\{x\},
\]
and choose vertices $b_i,c_i$ in two partite classes of $B_i$ different from $P_{B_i}(x)$.  By \Cref{lem:k6-rigid-shadow}, if $m_i$ is the median of $s_x,s_{b_i},s_{c_i}$, then the path from $m_i$ to $s_x$ has the form
\[
 m_i-u_i-s_x,
\]
and $s_{a_i}$ is adjacent to $u_i$.

If $u_i\ne u_j$ for some $i\ne j$, then
\[
 \dist_R(s_{a_i},s_{a_j})=4,
\]
which would create the cross-block edge $a_ia_j$.  Hence $u_1=u_2=u_3=:u$.  Now take $i\ne j$.  Since $m_j$ is adjacent to $u$ and $s_{b_j}$ is at distance two from $m_j$, the path from $s_{a_i}$ to $s_{b_j}$ has length
\[
 1+1+2=4.
\]
Thus $a_i b_j$ would be an edge, again impossible because the vertices lie in different blocks of the fan.  This contradiction proves $\rho_x(G)\le2$.
\end{proof}

\begin{corollary}[the rigid fan obstruction]
\label{cor:k6-rigid-fan}
For $r\ge3$, let $F_r$ be the graph obtained from blocks
\[
 B_i\cong K_{\{x,a_i\},\{b_i\},\{c_i\}}, \qquad 1\le i\le r,
\]
by identifying all copies of $x$ and otherwise keeping the blocks disjoint.  Then $F_r$ is not an exact $6$-leaf power, while $F_1$ and $F_2$ are exact $6$-leaf powers.
\end{corollary}

\begin{remark}[towards the complete $k=6$ compatibility condition]
Corollary~\ref{cor:k6-rigid-fan} shows that the obstruction at $k=6$ is not caused by complete multipartite blocks themselves, but by how their non-singleton partite classes meet at cutvertices.  Complete bipartite blocks remain flexible: one side may be placed at support distance one from a block hub and the other at distance three, giving cross-part distance four without forcing both sides to create short shadows at every cutvertex.  A natural complete classification should therefore be a part-compatibility condition on the block-cut tree: rigid incidences must be oriented or paired so that no cutvertex receives incompatible distance-two shadows from several incident blocks.  We leave this compatibility theorem as a boundary problem rather than using it as an assumption in the present exact-$5$ theory.
\end{remark}

\begin{problem}[the exact-six boundary for multipartite block graphs]\label{prob:k6-boundary}
Find a necessary and sufficient compatibility condition, on the block-cut tree and the partite classes at cutvertices, for a block-complete multipartite graph to be an exact $6$-leaf power.  Theorem~\ref{thm:k6-sufficient} gives a sufficient short-side condition, while Theorem~\ref{thm:k6-fan} shows that, in a one-cutvertex fan, two rigid shadows can be paired but three cannot.
\end{problem}

\section{Stable blow-ups and odd-distance boundary subclasses}\label{sec:blowups-boundaries}

The preceding subclass theorems are closed under the standard stable false-twin expansion.  We record this as a separate all-$k$ statement because it converts every skeleton classification into a blow-up classification.  The restriction $k\ge3$ is essential: for $k=2$, two leaves attached to the same support would become adjacent.

\begin{lemma}[component closure]\label{lem:components-k}
Let $k\ge2$.  A graph is an exact $k$-leaf power if and only if all of its connected components are exact $k$-leaf powers.
\end{lemma}

\begin{proof}
The only-if direction follows from closure under induced subgraphs.  Conversely, let the components be $G_1,\ldots,G_m$, and let $T_i$ be an exact $k$-leaf root of $G_i$ whenever $|V(G_i)|\ge2$.  If $|V(G_i)|\ge2$, choose an internal vertex $a_i$ of $T_i$.  If $G_i$ is an isolated vertex $v_i$, replace $T_i$ by the path from the graph leaf $v_i$ to a new attachment vertex $a_i$ of length $k+1$.

For $m\ge2$, take a new central vertex $c$ and join every $a_i$ to $c$ by an internally disjoint path of length $k+1$, except that in the isolated-vertex case the already chosen path from $v_i$ to $a_i$ is kept and $a_i$ is joined to $c$.  No graph leaf ceases to be a leaf, and no new non-graph leaf is created.  Distances between leaves in the same component are unchanged.  Distances between leaves in different components are larger than $k$, by the two connector segments.  Hence the resulting tree represents the disjoint union.
\end{proof}

\begin{definition}[stable blow-up]\label{def:stable-blowup}
Let $H$ be a graph.  A \emph{stable blow-up} of $H$ is obtained by replacing each vertex $v\in V(H)$ by a nonempty stable set $I_v$, and by adding all edges between $I_u$ and $I_v$ whenever $uv\in E(H)$.  We denote such a graph by $H[\{I_v\}_{v\in V(H)}]$.
\end{definition}

\begin{theorem}[stable blow-ups]\label{thm:stable-blow-up}
Let $k\ge3$, and let $G$ be a stable blow-up of a graph $H$.  Then
\[
 G\text{ is an exact }k\text{-leaf power}
 \quad\Longleftrightarrow\quad
 H\text{ is an exact }k\text{-leaf power}.
\]
\end{theorem}

\begin{proof}
First suppose that $H$ has an exact $k$-leaf root $T$.  If $H$ has just one vertex, then its blow-up is edgeless; it is represented by the one-vertex tree when the bag has size one, and otherwise by a star whose leaves are the bag elements, since all pairwise leaf distances are $2\ne k$.  Otherwise every graph leaf of $T$ has a neighbour.  For a vertex $v\in V(H)$, let $a_v$ be the neighbour of the leaf $v$ in $T$.  Delete the leaf $v$ and attach all vertices of the stable bag $I_v$ as new leaves adjacent to $a_v$.  Doing this independently for all $v\in V(H)$ gives a tree $T'$.  If $x\in I_u$ and $y\in I_v$ with $u\ne v$, then
\[
 \dist_{T'}(x,y)=\dist_T(u,v),
\]
so $xy$ is an edge exactly when $uv\in E(H)$.  If $x,y$ lie in the same bag, then $\dist_{T'}(x,y)=2\ne k$, because $k\ge3$; hence each bag remains stable.  Thus $T'$ is an exact $k$-leaf root of $G$.

Conversely, choose one representative $r_v\in I_v$ from each bag.  The induced subgraph of $G$ on $\{r_v:v\in V(H)\}$ is isomorphic to $H$.  Exact $k$-leaf powers are closed under induced subgraphs: in an exact root, take the minimal subtree spanning the selected leaves.  Its leaf set is exactly the selected set, and all distances between selected leaves are unchanged.  Therefore $H$ is an exact $k$-leaf power.
\end{proof}

\begin{corollary}[blown-up skeleton classes]\label{cor:blown-up-skeletons}
The following hold for arbitrary nonempty stable blow-ups.
\begin{enumerate}[label=\textup{(\roman*)}]
\item A stable blow-up of a chordless cycle $C_\ell$ is an exact $5$-leaf power if and only if $\ell$ is even.
\item A stable blow-up of the ladder $L_t$ is an exact $5$-leaf power if and only if $t\le2$.
\item A stable blow-up of a cactus $F$ is an exact $5$-leaf power if and only if $F$ is bipartite.
\item For every $k\ge3$, a stable blow-up of a complete multipartite graph $K_{n_1,\ldots,n_s}$ is an exact $k$-leaf power if and only if $k$ is even or $s\le2$.
\item If $H$ is block-complete multipartite, then a stable blow-up of $H$ is an exact $5$-leaf power if and only if $H$ is bipartite.
\end{enumerate}
\end{corollary}

\begin{proof}
Apply \Cref{thm:stable-blow-up} to the corresponding skeleton theorem: the cycle block-language theorem and the finite cases $C_4,C_6$ for (i), \Cref{thm:ladder} for (ii), \Cref{cor:cacti} for (iii), \Cref{thm:complete-multipartite-all-k} for (iv), and \Cref{thm:block-complete-multipartite} for (v).
\end{proof}

Odd exact distances impose strong restrictions on graph classes that are close to chordal.  The next statements are useful boundary cases: in chordal classes only forests survive, while in cographs the only obstruction is non-bipartiteness.

\begin{theorem}[chordal graphs at odd exact distances]\label{thm:chordal-odd}
Let $k\ge3$ be odd, and let $G$ be a chordal graph.  Then
\[
 G\text{ is an exact }k\text{-leaf power}
 \quad\Longleftrightarrow\quad
 G\text{ is a forest}.
\]
\end{theorem}

\begin{proof}
If $G$ is an exact $k$-leaf power with $k$ odd, then $G$ is bipartite: a bipartition of the root tree separates all leaf pairs at odd distance.  A chordal bipartite graph is acyclic, since a shortest cycle in any graph is induced and a bipartite cycle has length at least four.  Thus $G$ is a forest.

Conversely, every tree is an exact $3$-leaf power: attach one new leaf to every vertex of the tree.  By \Cref{lem:components-k}, forests are handled componentwise.  Repeatedly subdividing every edge incident with a graph leaf once raises the exact distance by two and preserves the represented graph.  Hence every forest is an exact $k$-leaf power for every odd $k\ge3$.
\end{proof}

\begin{corollary}[standard chordal subclasses]\label{cor:standard-chordal-subclasses}
Let $k\ge3$ be odd.  In each of the following classes, exact $k$-leaf powers are exactly the forests:
\[
 \text{interval graphs},\qquad
 \text{proper interval graphs},\qquad
 \text{split graphs},\qquad
 \text{threshold graphs}.
\]
\end{corollary}

\begin{proof}
All four classes are subclasses of chordal graphs, so the result follows from \Cref{thm:chordal-odd}.
\end{proof}

\begin{theorem}[cographs at odd exact distances]\label{thm:cograph-odd}
Let $k\ge3$ be odd, and let $G$ be a cograph.  Then
\[
 G\text{ is an exact }k\text{-leaf power}
 \quad\Longleftrightarrow\quad
 G\text{ is bipartite}.
\]
\end{theorem}

\begin{proof}
The forward implication is again the parity bipartition of the exact root.

For the converse, assume that $G$ is a bipartite cograph.  It is enough to consider a connected component $C$ with at least two vertices.  A connected cograph is $P_4$-free.  If $C$ had a nonedge $xy$ across its bipartition, then every shortest $x$--$y$ path would be induced and would have odd length at least three.  Its first four vertices would induce a $P_4$, a contradiction.  Therefore every pair of vertices in opposite bipartition classes is adjacent, and $C$ is complete bipartite.  Complete bipartite graphs are exact $k$-leaf powers for every odd $k\ge3$ by \Cref{thm:complete-multipartite-all-k}.  The disconnected case follows from \Cref{lem:components-k}.
\end{proof}

\begin{remark}[where to go beyond these subclasses]
The chordal and cograph theorems are deliberately sharp boundary statements.  They should not be confused with a classification of all bipartite $P_5$-free or all chain graphs: those classes allow nested-neighbourhood blocks that are neither complete bipartite nor forest-like, and their exact $5$ behaviour is a separate problem.  The full compatibility theorem for the even boundary $k=6$ in block-complete multipartite decompositions and the exact $5$ classification of theta graphs are natural next test cases.
\end{remark}

\section{Chordless cycles: an explicit terminal block language}\label{sec:cycle-correction}

This section contains the paper's most detailed root-classification theorem.  We use a word language rather than a purely pictorial list of tree pieces.  The reason is structural: once the heavy spine is fixed, the exact-distance condition becomes the local equation
\[
 r+s+|i-j|=5,
\]
and endpoint configurations are recorded explicitly.  Thus completeness is proved by a finite transition argument rather than by a separate enumeration of small drawings.

We give a finite block-language classification for exact $5$-leaf roots of chordless cycles of length at least eight.  The notation below separates the normal $23$-terminal language
\[
  23\,\mathcal B^{*}\,\widehat{\mathcal B}
\]
from the complementary terminal-$123$ language.  The latter consists of the cases in which endpoint type $123$ occurs at one or both ends of the heavy spine.  Thus the result is not merely an admissibility test of the form ``a word is valid if its distance graph is a cycle''.  It is a finite explicit block language, and the two terminal sublanguages together account for all endpoint configurations.

\subsection*{Spine words}

For a finite sequence
\[
 W=A_0|A_1|\cdots|A_m,
 \qquad A_i\subseteq\{1,2,3\},
\]
define a tree $T(W)$ as follows.  Take a spine path
\[
 x_0x_1\cdots x_m.
\]
For every $r\in A_i$, attach at $x_i$ a pendant path of length $r$ and denote its leaf by $z_{i,r}$.  Thus
\[
 L(T(W))=\{z_{i,r}:0\le i\le m,\ r\in A_i\}.
\]
The number of leaves is
\[
 \lambda(W)=\sum_{i=0}^{m}|A_i|.
\]
For two leaves of $T(W)$ one has
\begin{equation}\label{eq:cycle-distance-word}
 d_{T(W)}(z_{i,r},z_{j,s})=r+s+|i-j|.
\end{equation}
Hence the exact $5$-leaf graph of $T(W)$ has vertex set
\[
 \Omega(W)=\{(i,r):r\in A_i\}
\]
and adjacency rule
\begin{equation}\label{eq:cycle-adj-word}
 (i,r)(j,s)\in E\quad\Longleftrightarrow\quad r+s+|i-j|=5.
\end{equation}

We use the abbreviations
\[
0=\varnothing,\quad
1=\{1\},\quad
12=\{1,2\},\quad
13=\{1,3\},\quad
23=\{2,3\},\quad
123=\{1,2,3\}.
\]
The symbols $2$, $3$ and the internal symbols $23,123$ will not occur in the canonical cycle words below.

If positions outside the word are interpreted as empty, then \eqref{eq:cycle-adj-word} says that every leaf has degree two exactly when the following three local equations hold at every position $i$ at which the indicated length occurs:
\begin{align}
1\in A_i:
&\quad
\mathbf 1_{3\in A_{i-1}}+\mathbf 1_{3\in A_{i+1}}
+\mathbf 1_{2\in A_{i-2}}+\mathbf 1_{2\in A_{i+2}}
+\mathbf 1_{1\in A_{i-3}}+\mathbf 1_{1\in A_{i+3}}=2,
\tag{D1}\label{eq:deg-one}\\
2\in A_i:
&\quad
\mathbf 1_{3\in A_i}
+\mathbf 1_{2\in A_{i-1}}+\mathbf 1_{2\in A_{i+1}}
+\mathbf 1_{1\in A_{i-2}}+\mathbf 1_{1\in A_{i+2}}=2,
\tag{D2}\label{eq:deg-two}\\
3\in A_i:
&\quad
\mathbf 1_{2\in A_i}
+\mathbf 1_{1\in A_{i-1}}+\mathbf 1_{1\in A_{i+1}}=2.
\tag{D3}\label{eq:deg-three}
\end{align}
These equations are simply the list of all solutions of $r+s+|i-j|=5$ for $r\in\{1,2,3\}$.

\subsection*{The block language}

Define the five normal return blocks
\[
\begin{array}{rclcrcl}
B_2&=&1|1|0,&&
B_3&=&12|0,\\[1mm]
B_4&=&1|1|1|1|12|1|1|1|1|0,&&
B_5&=&1|13|1|0,\\[1mm]
B_6&=&1|1|1|13|1|1|1|0.&&&
\end{array}
\]
Let
\[
 \mathcal B=\{B_2,B_3,B_4,B_5,B_6\}.
\]
Whenever a block $U$ ends in $0$, write $\widehat U$ for the word obtained from $U$ by replacing its final $0$ by $23$, and write
\[
 \widehat{\mathcal U}=\{\widehat U:U\in\mathcal U\}.
\]
The normal language is
\begin{equation}\label{eq:normal-cycle-language}
 \mathcal N=23\,\mathcal B^{*}\,\widehat{\mathcal B}.
\end{equation}
This is the Fig.~15-type normal language of \cite{BLR2010} in spine-word notation.

The terminal-$123$ roots are generated by the following terminal-$123$ blocks:
\[
\begin{array}{rclcrcl}
L_1&=&123|13|1|1|1|0,&&
L_2&=&123|1|12|1|1|1|1|0,\\[1mm]
R_1&=&1|1|1|13|123,&&
R_2&=&1|1|1|1|12|1|123,
\end{array}
\]
and by the two short double-terminal bridges
\[
D_1=123|13|123,
\qquad
D_2=123|1|12|1|123.
\]
Put
\[
 \mathcal L=\{L_1,L_2\},
 \qquad
 \mathcal R=\{R_1,R_2\},
 \qquad
 \mathcal D=\{D_1,D_2\}.
\]
The terminal-$123$ language is
\begin{equation}\label{eq:exceptional-cycle-language}
\mathcal X=
\widehat{\mathcal L}
\cup
\mathcal L\mathcal B^{*}\widehat{\mathcal B}
\cup
23\mathcal B^{*}\mathcal R
\cup
\mathcal D
\cup
\mathcal L\mathcal B^{*}\mathcal R.
\end{equation}
Thus $\widehat{\mathcal L}$ consists of the one-sided terminal-$123$ words which end normally, $23\mathcal B^{*}\mathcal R$ consists of the right terminal-$123$ words, $\mathcal L\mathcal B^{*}\widehat{\mathcal B}$ consists of the left terminal-$123$ words with a normal right end, $\mathcal D$ consists of the two short words with both ends of type $123$, and $\mathcal L\mathcal B^{*}\mathcal R$ consists of the long words with both ends of type $123$.

\begin{theorem}[explicit block language for chordless cycles]\label{thm:cycle-block-language}
Let $l\ge8$ be even.  A tree $T$ is an exact $5$-leaf root of $C_l$ if and only if, after possibly reversing the spine,
\[
 T\cong T(W)
\]
for some word
\[
 W\in\mathcal N\cup\mathcal X
\]
with $\lambda(W)=l$.
Equivalently, the words in \eqref{eq:normal-cycle-language} are precisely the normal $23$-terminal block roots, and the words in \eqref{eq:exceptional-cycle-language} are precisely the terminal-$123$ block roots.
\end{theorem}

\begin{proof}
We first recall the standard spine reduction for exact $5$-roots of long chordless cycles.  Let $T$ be an exact $5$-leaf root of $C_l$, $l\ge8$.  For every vertex $u$ of $T$, at most two components of $T-u$ contain at least two leaves; this is the separation statement proved in Lemma~7(i) of \cite{BLR2010}.  Call an edge heavy if both components of $T-e$ contain at least two leaves.  The heavy edges form a nonempty path $P=x_0x_1\cdots x_m$: they are connected because every edge between two heavy edges is heavy, and no three of them can meet at a vertex by the separation statement.  Every component of $T-V(P)$ contains exactly one leaf, otherwise the edge joining that component to $P$ would be heavy.  Hence $T$ is obtained from $P$ by attaching pendant leaf paths.  Since every vertex of $C_l$ has degree two, no pendant leaf path can have length at least four: length at least five gives no exact-$5$ neighbour, and length four gives at most the unique possible length-one neighbour at the same spine vertex.  Also, two pendant paths of the same length at the same spine vertex would give false twins in $C_l$.  Therefore $T=T(W)$ for a unique spine word $W=A_0|\cdots|A_m$ with $A_i\subseteq\{1,2,3\}$, up to reversal, and $|A_0|,|A_m|\ge2$.

The rest of the proof is the explicit finite forcing calculation for such words.  The degree-two conditions are exactly \eqref{eq:deg-one}--\eqref{eq:deg-three}.  Applying them at the two ends and then at the first unresolved internal position gives two immediate restrictions: the only possible end symbols are $23$ and $123$, and the only possible internal symbols are
\[
 0,
 \quad 1,
 \quad 12,
 \quad 13.
\]
Indeed, an end symbol $13$ violates \eqref{eq:deg-three}, while an end symbol $12$ forces successively a right-neighbour containing $2$ and a second-neighbour containing $1$, and then violates \eqref{eq:deg-two} at the first internal $2$-leaf.  Similarly, an internal symbol containing $2$ but not $1$ or containing $3$ without $1$ violates \eqref{eq:deg-two} or \eqref{eq:deg-three} after at most two further applications of the same equations; an internal symbol $23$ or $123$ would create a second heavy endpoint, contradicting the maximality of the chosen heavy path.  Thus the displayed alphabet is canonical.

Now mark the boundary symbols
\[
 0,
 \quad 23,
 \quad 123.
\]
Starting at a boundary symbol and reading to the next boundary symbol, the equations \eqref{eq:deg-one}--\eqref{eq:deg-three} force exactly the following first-return table.  The entry in row $\sigma$ and column $\tau$ lists all possible words following $\sigma$ up to and including the next boundary symbol $\tau$.
\[
\begin{array}{c|c|l}
\text{left boundary }\sigma & \text{next boundary }\tau & \text{possible first-return words}\\
\hline
0\text{ or }23 & 0
 & B_2,
 B_3,
 B_4,
 B_5,
 B_6\\
0\text{ or }23 & 23
 & \widehat B_2,
 \widehat B_3,
 \widehat B_4,
 \widehat B_5,
 \widehat B_6\\
0\text{ or }23 & 123
 & R_1,
 R_2\\
123 & 0
 & L_1,
 L_2\\
123 & 23
 & \widehat L_1,
 \widehat L_2\\
123 & 123
 & D_1,
 D_2.
\end{array}
\tag{T}
\]
We spell out why this table is exhaustive.  Suppose first that the left boundary is $23$ or $0$.  If the next symbol contains $2$, then \eqref{eq:deg-two} and \eqref{eq:deg-one} force the branch $12|0$ or its terminal version $12|23$.  If the next symbol is $13$, then \eqref{eq:deg-three} cannot be satisfied.  Thus the only remaining possibility begins with $1$.  The next application of \eqref{eq:deg-one} gives either $1|1$, $1|13$, or the terminal continuation that later becomes $R_1$ or $R_2$.  The branch $1|13$ is forced to $1|13|1$ and then must stop at $0$ or $23$, giving $B_5$ or $\widehat B_5$.  The branch $1|1$ either stops as $1|1|0$ or $1|1|23$, giving $B_2$ or $\widehat B_2$, or continues with another $1$.  Continuing the equations from $1|1|1$ gives exactly two longer branches: the $13$-branch
\[
 1|1|1|13|1|1|1|0,
 \quad
 1|1|1|13|1|1|1|23,
 \quad
 1|1|1|13|123,
\]
which gives $B_6$, $\widehat B_6$ and $R_1$, and the $12$-branch
\[
 1|1|1|1|12|1|1|1|1|0,
 \quad
 1|1|1|1|12|1|1|1|1|23,
 \quad
 1|1|1|1|12|1|123,
\]
which gives $B_4$, $\widehat B_4$ and $R_2$.  Every other continuation makes one of \eqref{eq:deg-one}--\eqref{eq:deg-three} equal to $0$, $1$, or at least $3$ before the next boundary is reached.

If the left boundary is $123$, then \eqref{eq:deg-three} forces the next symbol to contain $1$.  If the next symbol is $13$, the equations force either the short terminal branch
\[
 13|123
\]
or the branch
\[
 13|1|1|1|0
\]
and its terminal normal variant
\[
 13|1|1|1|23.
\]
These are $D_1$, $L_1$ and $\widehat L_1$.  If the next symbol is $1$, then \eqref{eq:deg-two} at the left endpoint forces a following $12$, and the equations force either
\[
 1|12|1|123
\]
or
\[
 1|12|1|1|1|1|0
\]
and its terminal normal variant
\[
 1|12|1|1|1|1|23.
\]
These are $D_2$, $L_2$ and $\widehat L_2$.  No other first symbol after $123$ satisfies the end equations.  This proves the table.

We now prove necessity.  The first boundary of a cycle word is $23$ or $123$, and the last boundary is again $23$ or $123$.  If the first boundary is $23$, repeatedly applying the first three rows of table (T) shows that the word is either
\[
 23\mathcal B^{*}\widehat{\mathcal B}
\]
or
\[
 23\mathcal B^{*}\mathcal R.
\]
The first alternative is the normal language $\mathcal N$, and the second is the right terminal-$123$ part of $\mathcal X$.  If the first boundary is $123$, then the last three rows of (T) give three possibilities.  The word may end immediately at $23$, giving $\widehat{\mathcal L}$; it may end immediately at $123$, giving $\mathcal D$; or it may first pass through a $0$ via a block in $\mathcal L$, after which the first three rows give either $\mathcal B^{*}\widehat{\mathcal B}$ or $\mathcal B^{*}\mathcal R$.  These are respectively
\[
 \mathcal L\mathcal B^{*}\widehat{\mathcal B}
 \quad\text{and}\quad
 \mathcal L\mathcal B^{*}\mathcal R.
\]
Thus every exact $5$-leaf root of $C_l$ gives a word in $\mathcal N\cup\mathcal X$.

Conversely, let $W\in\mathcal N\cup\mathcal X$ with $\lambda(W)=l$.  Table (T), read from left to right, is exactly the list of first-return words for which the equations \eqref{eq:deg-one}--\eqref{eq:deg-three} are satisfied.  Therefore every leaf of $T(W)$ has degree two in the exact $5$-leaf graph.  The same table also tracks the two cycle strands through each block: a normal return block sends the two strands entering the left boundary to the two strands leaving the right boundary; a terminal $23$ closes them normally; a terminal $123$ closes them through one of $R_1,R_2,D_1,D_2$; and the left blocks $L_1,L_2$ start the same two strands from the terminal-$123$ left endpoint.  Thus no proper subword closes a cycle by itself, and the exact $5$-leaf graph is connected.  It is therefore a connected $2$-regular graph on $\lambda(W)=l$ vertices, hence it is $C_l$.  By \eqref{eq:cycle-distance-word}, $T(W)$ is an exact $5$-leaf root of $C_l$.

The heavy path $P$ was intrinsically defined, so the spine word is unique up to reversing $P$.  This proves the theorem.
\end{proof}

\begin{example}[terminal-$123$ roots]
The shortest terminal-$123$ word is
\[
 D_1=123|13|123.
\]
It has eight leaves and its exact $5$-leaf graph is the cycle
\[
 z_{0,1},z_{1,3},z_{2,1},z_{0,2},z_{0,3},z_{1,1},z_{2,3},z_{2,2},z_{0,1}
\]
after cyclic relabelling.  The second short double-terminal bridge is
\[
 D_2=123|1|12|1|123,
\]
which gives a terminal-$123$ exact $5$-leaf root of $C_{10}$.  Longer terminal-$123$ roots are obtained by inserting any word from $\mathcal B^{*}$ between the left and right terminal-$123$ blocks, or by using only one terminal-$123$ block and ending normally on the other side.
\end{example}

\section{Concluding remarks}

The results of the paper give a structural picture of exact leaf powers on several concrete graph classes.  The most detailed root-level classification is the terminal block-language theorem for chordless cycles: all exact $5$-leaf roots of $C_l$, $l\ge8$, are described by the union of the normal terminal-$23$ language and the complementary terminal-$123$ language.  The remaining exact-distance-five theorems show that the same support-distance mechanisms also separate several distinct behaviours.  Sparse square chains are obstructed after two squares, yielding the ladder dichotomy and the induced square-chain obstruction.  Complete bipartite blocks behave in the opposite way: all induced squares inside one dense block are carried by a single support edge, and block-complete multipartite graphs are exact $5$ precisely in the bipartite case.  Bipartite co-cluster graphs show that dense support gadgets also represent structured missing-edge patterns such as crowns.

The larger-distance results show that the dense-block picture is controlled by parity.  Complete multipartite graphs are exact $k$ exactly when $k$ is even or the graph is bipartite.  For block-complete multipartite graphs, odd $k\ge5$ again gives the bipartite members, whereas even $k\ge8$ admits all complete multipartite blocks.  The remaining boundary $k=6$ is genuinely different: short same-part shadows may collide through cutvertices.  The orientation theorem and the exact-six fan theorem isolate the first positive and negative compatibility phenomena at this boundary.

Two problems remain particularly natural.  First, find the full centre-pairing compatibility condition for exact $6$-leaf representability of block-complete multipartite graphs.  Second, extend the explicit exact-$5$ root-language classification beyond cycles and cacti to theta graphs or larger outerplanar subclasses.  Both directions preserve the support-tree viewpoint of this paper while moving beyond the named subclasses treated here.

\section*{Acknowledgements}
This work was supported by the Young Scientists Fund of the National Natural Science Foundation of China (No. 20201191803).

\section*{Conflict of interest}
The authors declare no conflict of interest.

\section*{Declaration on AI-assisted technologies}
During manuscript preparation, the authors used OpenAI's ChatGPT for language polishing, organization of LaTeX drafts, and assistance with illustrative TikZ code.  The authors reviewed and edited the content and take full responsibility for the manuscript.

\end{document}